\documentclass[12pt]{article}

\usepackage{amsthm}
\usepackage{amssymb}
\usepackage{amsmath}
\usepackage{comment}
\usepackage{url} 
\usepackage[noabbrev,capitalise]{cleveref}

\usepackage{color}
\usepackage[normalem]{ulem}

\usepackage[margin=1in]{geometry}

\newcommand{\eps}{\varepsilon}
\newcommand{\mc}[1]{\mathcal{#1}}

\newcommand{\mf}[1]{\mathfrak{#1}}
\newcommand{\bb}[1]{\mathbb{#1}}
\newcommand{\brm}[1]{\operatorname{#1}}

\theoremstyle{definition}

\theoremstyle{plain}
\newtheorem{thm}{Theorem}[section]
\newtheorem{lem}[thm]{Lemma}

\newtheorem{cor}[thm]{Corollary}

\newtheorem{obs}[thm]{Observation} 

\newtheorem{conj}[thm]{Conjecture}

\newcommand{\h}{h}
\newcommand{\Forb}{\operatorname{Forb}}

\newcommand{\wpn}{\chi_c}

\newcommand{\crit}{\operatorname{dang}}
\newcommand{\dang}{\operatorname{dang}}
\newcommand{\red}{\operatorname{red}}

\title{Typical structure of hereditary graph families. II. Exotic examples}
\author{
Sergey Norin\thanks{Department of Mathematics and Statistics, McGill University. Email: {\tt sergey.norin@mcgill.ca}. Supported by an NSERC Discovery grant.}\and \and
Yelena Yuditsky\thanks{Department of Mathematics, Ben-Gurion University of the Negev.  Email: {\tt yuditskyl@gmail.com}.}}

\begin{document}

\maketitle

\begin{abstract}
A graph $G$ is \textit{$H$-free} if it does not contain an induced subgraph isomorphic to $H$. The study of the typical structure of $H$-free graphs was initiated by Erd\H{o}s, Kleitman and Rothschild~\cite{EKR76}, who have shown that almost all $C_3$-free graphs are bipartite. Since then the typical structure of $H$-free graphs has been determined for several families of graphs $H$, including complete graphs, trees and cycles. Recently,  Reed and Scott~\cite{RS18} proposed a conjectural description of the typical structure of $H$-free graphs for all graphs $H$, which extends all previously known results in the area.

We construct an infinite family of graphs for which the Reed-Scott conjecture fails, and use  the methods we developed in the prequel paper~\cite{NorYud191} to describe the typical structure of $H$-free graphs for graphs $H$ in this family. 

Using  similar techniques, we construct an infinite family of graphs $H$ for which the maximum size of a homogenous set in a typical $H$-free graph is sublinear in the number of vertices, answering a question of Loebl et al.~\cite{LRSTT10} and Kang et al.~\cite{KMRS14}.  
\end{abstract}

\section{Introduction}

Let $\mc{F}$ be a  family of graphs. We say that $\mc{F}$ is \emph{hereditary} if $\mc{F}$ is closed under isomorphism and taking induced subgraphs. 
Let $H$ be a graph, we say that a graph $G$ is \textit{$H$-free} if it does not contain an induced subgraph isomorphic to $H$, and we denote by $\brm{Forb}(H)$ the family of all $H$-free graphs. Clearly, $\brm{Forb}(H)$ is hereditary. More generally, if $\mc{H}$ is a collection of graphs we say that a graph $G$ is \textit{$\mc{H}$-free} if $G$ is $H$-free for every $H \in \mc{H}$, and we denote by $\brm{Forb}(\mc{H})$ the family of all $\mc{H}$-free graphs.

In this paper we study the typical structure of graphs in $\brm{Forb}(H)$. Our main results are constructions of graphs $H$ for which this structure is more complex than in the previously known examples. 

For a family of graphs $\mc{F}$, let $\mc{F}^n$ denote the set of graphs in $\mc{F}$ with vertex set $[n]=\{1,2,\ldots,n\}$. 
We say that a property $\mc{P}$ holds for \emph{almost all graphs} in $\mc{F}$, if $$\lim_{n \to \infty}\frac{|\mc{F}^n \cap \mc{P}|}{|\mc{F}^n|} = 1.$$ 

The study of the  typical structure of graphs in $\brm{Forb}(H)$ was initiated by Erd\H{o}s, Kleitman and Rothschild~\cite{EKR76}.  They have shown that almost all $C_3$-free graphs are bipartite. Pr{\"o}mel and Steger~\cite{PS91, PS92} obtained a structural characterization of typical $C_4$ and $C_5$-free graphs. They have shown that vertices of almost all $C_4$-free graphs can be partitioned into a clique and a stable set,\footnote{I.e. almost all $C_4$-free graphs are split graphs.} and that for almost all $C_5$-free graph $G$ either the vertices of $G$ can be partitioned into a clique and a set inducing a disjoint union of cliques, or the
vertices of $G$ can be partitioned into a stable set and a set inducing a complete multipartite graph.

Recently, Reed and Scott~\cite{RS18} proposed a conjecture which informally states that a similar description of almost all $H$-free graphs is possible for any $H$. To state it precisely we need a few  definitions. A \emph{pattern} $\mf{F}=(\mc{F}_1,\ldots,\mc{F}_k)$ is a finite collection of hereditary families of graphs.  We say that a partition $\mc{P}=(P_1,P_2,\ldots,P_k)$ of the vertex set of a graph $G$ is an \emph{$\mf{F}$-partition} for a pattern $\mf{F}$ as above if $G[P_i] \in \mc{F}_i$ for every $i \in [k]$.  Let $\mc{P}(\mf{F})$ denote the family of all graphs admitting an $\mf{F}$-partition.  If $\mf{F}$ is a pattern with $|\mf{F}|=l$ such that every element of $\mf{F}$ isthe same family $\mc{F}$ then we refer to an $\mf{F}$-partition as  an \emph{$(\mc{F},l)$-partition} and denote $\mc{P}(\mf{F})$ by $\mc{P}(\mc{F},l)$.

A pattern $\mf{F}$ is \emph{$H$-free} if $H \not \in \mc{P}(\mf{F})$. 
We say that $\mf{F}$ is  \emph{sharply $H$-free} if it is $H$-free and for every $\mc{F} \in \mf{F}$ every minimal graph $J$ which is not in $\mc{F}$\footnote{That is every $J \not \in \mc{F}$ such that $J \setminus v \in \mc{F}$ for every $v \in \mc{F}$.} is isomorphic to a subgraph of $H$. It is easy to see that every maximal $H$-free pattern is sharply $H$-free, and so we restrict our attention to sharply $H$-free patterns. Note that, conveniently, for any graph $H$ and any integer $k$ there are finitely many sharply $H$-free patterns of size $k$, as every element of such pattern is completely determined by the collection of the subgraphs of $H$ which belong to it.
 
A structural description of $\brm{Forb}(H)$ along the lines of the results of~\cite{EKR76,PS91,PS92} in the language we have just introduced can now be stated as follows. 
For almost every $G \in \brm{Forb}(H)$ there exists a pattern $\mf{F}$ such that 
\begin{itemize}
	\item[(S1)] $\mf{F}$ is sharply $H$-free,
	\item[(S2)] $G \in \brm{P}(\mf{F})$,
	\item[(S3)] elements of $\mf{F}$ are ``structured".
\end{itemize} 

Let $\mc{C}=\brm{Forb}(\bar{K_2})$ denote the family of all complete graphs and let $\mc{S}=\brm{Forb}(K_2)$ denote the family of all edgeless graphs. Then the result of~\cite{EKR76} shows that
$\mf{F}=(\mc{S},\mc{S})$ satisfies the above conditions for $H=C_3$, while for $H=C_5$ we need to take to  $\mf{F}=(\mc{C},\brm{Forb}(P_3))$ or  $\mf{F}=(\mc{S},\brm{Forb}(\bar{P}_3))$, depending on $G$.

Note that $\mf{F}=(\brm{Forb}(H))$ trivially satisfies conditions (S1) and (S2), but does not give any insight in the structure of $\brm{Forb}(H)$. Thus we need to formalize condition (S3). We say that a pattern $\mf{F}$  is \emph{proper} if $\mc{F}^{n}\neq \emptyset$ for every $\mc{F} \in \mf{F}$ and every positive integer $n$. By Ramsey's theorem the above condition is equivalent to the requirement that for every $\mc{F} \in \mf{F}$ either $\mc{C} \subseteq \mc{F}$ or $\mc{S} \subseteq \mc{F}$ . 

For a pair of non-negative integers $s$ and $t$ let $\mf{F}(s,t)=(\mc{S},\ldots,\mc{S},\mc{C},\ldots,\mc{C})$ denote the proper pattern consisting of $s$ families of edgeless graphs and $t$ families of complete graphs, and let $\mc{H}(s,t)$ denote $\mc{P}(\mf{F}(s,t))$. Thus  $\mc{H}(s,t)$ is a family of all graphs whose vertex set can be partitioned into $s$ stable sets and $t$ cliques. For every proper pattern $\mf{F}$ we have $\mc{H}(s,t) \subseteq \mf{F}$ for some $s,t$ such that $s+t=|\mf{F}|$, and thus the following observation holds.
 
\begin{obs}\label{obs:colouring} For a graph $H$ and a positive integer $l$ the following are equivalent.
	\begin{itemize}
		\item  there exist non-negative integers $s$ and $t$ with $s+t=l$ such that $H \not \in \mc{H}(s,t)$,  
		\item there exists a proper $H$-free pattern $\mf{F}$ such that $|\mf{F}|=l$.
	\end{itemize}
\end{obs} 

The maximum integer $l$ satisfying the conditions of Observation~\ref{obs:colouring} for a graph $H$ is called \emph{the witnessing partition number of $H$} and is denoted be $\wpn(H)$. One can impose meaningful structure on the elements of a proper $H$-free pattern $\mf{F}$ by insisting simply that it has maximum possible size, i.e. 
 $|\mf{F}|=\wpn(H)$.  Combining conditions (S1),(S2) and (S3), we say that a pattern $\mf{F}$ is a \emph{clean $H$-free pattern} if  $\mf{F}$ is proper, sharply $H$-free, and    $|\mf{F}|=\wpn(H)$. We say that a 
 clean $H$-free pattern is a \emph{clean $H$-free profile} for a graph $G$ if $G \in \brm{P}(\mf{F})$.
 We can now precisely state the Reed-Scott's conjecture mentioned above.
 
 \begin{conj}\label{c:RS}
 For every graph $H$, almost every $H$-free graph has a clean $H$-free profile. 
 \end{conj}
 
Conjecture~\ref{c:RS} has been verified for cliques \cite{EKR76}, cycles \cite{PS91, PS92, BB11,RS18}, trees \cite{RY18} and \textit{critical graphs} \cite{BB11}. 
 
Our first main result shows that  Conjecture~\ref{c:RS} is false in general.
  
\begin{thm}\label{t:ARS} 
 	There exists infinitely many graphs  $H$ such that almost every $H$-free graph has no clean $H$-free profile. 
\end{thm}
 
The proof of Theorem~\ref{t:ARS} can be vaguely outlined as follows.
The graphs $H$ satisfying the theorem are constructed  so that clean $H$-free patterns are severely restricted. To do this we ensure that $H$ admits a variety of partitions into graphs with simple structure. The vertex sets of parts of these partitions are chosen at random to further guarantee that $H$ does not admit the partitions into ``simple" graphs, except for the ones we specifically prescribed. 

\vskip 10pt
Using variants of this technique, we are able to generate examples of graphs $H$ such that almost all $H$-free graphs have given structure for a fairly wide variety of specifications.  Our second class of examples constructed this way answers a question from \cite{LRSTT10,KMRS14} related to the famous Erd\H{o}s-Hajnal conjecture, which we now state.

A \textit{homogenous set} in a graph is either an independent set or a clique. We denote by $h(G)$ the size of the largest homogenous set in a graph $G$. Erd\H{o}s and Hajnal made the following conjecture. 

\begin{conj} [Erd\H{o}s-Hajnal Conjecture\cite{EH89}]
For every graph $H$, there exists an $\varepsilon=\varepsilon(H)>0$ such that all $H$-free graphs $G$ have $\h(G)\ge |V(G)|^{\varepsilon}$. 
\end{conj}

The conjecture appears to be very hard and is known to hold only for a few graphs $H$, see ~\cite{C14} for a survey. A way to relax the  conjecture in line with the subject of this paper is to consider almost all $H$-free graphs for a given graph $H$. Loebl et al.~\cite{LRSTT10} did just this proving the following.

\begin{thm}[Loebl et al. \cite{LRSTT10}]
For every graph $H$, there exists an $\varepsilon=\varepsilon(H)>0$ such that almost all $H$-free graphs $G$ have $h(G)\ge |V(G)|^{\varepsilon}$. 
\end{thm}

Kang et al. \cite{KMRS14} have shown  that a stronger conclusion holds for almost all graphs $H$. We say that a graph $H$  has the \textit{asymptotic linear Erd\H{o}s-Hajnal property} if there exists $b>0$ such that  almost all $H$-free graphs $G$ satisfy $h(G)\ge b |V(G)|$. 

\begin{thm}[Kang et al \cite{KMRS14}]\label{thmKMRS}
Almost all graphs have the asymptotic linear Erd\H{o}s-Hajnal property.
\end{thm}

It is mentioned in \cite{LRSTT10} and \cite{KMRS14} that $P_3$, the path on $3$ vertices, does not have the asymptotic linear Erd\H{o}s-Hajnal property. More precisely, as a direct corollary to a result Aleksandrovskii (cf. \cite{Y95}), the authors of~\cite{LRSTT10,KMRS14} observed the following. 

\begin{obs}
	\label{o:P3homogenous}
Almost all $P_3$-free graphs $G$ have $h(G)=\Theta\left(\frac{|V(G)|}{\log |V(G)|}\right)$.
\end{obs}

The authors in \cite{KMRS14} and  \cite{LRSTT10} asked if there exists graphs other than $P_3$ and, possibly, $P_4$ that do not have the asymptotic linear Erd\H{o}s-Hajnal property. We answer this question affirmativelty.

\begin{thm}\label{t:P3vague}
There exist infinitely many graphs which do not have the asymptotic linear Erd\H{o}s-Hajnal property.
\end{thm}

We also present a third class of examples of similar nature to the classes appearing in Theorems~\ref{t:ARS} and~\ref{t:P3vague}, but we postpone its description to the next section, as it requires more preparation to motivate.
	
Showing that the graphs $H$ in the families that we construct have the claimed properties requires us to analyze the structure of typical $H$-free graphs. In Section~\ref{s:prelim} we present the tools for such analysis, which we developed in~\cite{NorYud191}.  In Section~\ref{s:applications} we explicitly describe our families of exotic examples, including the families satisfying Theorems~\ref{t:ARS} and~\ref{t:P3vague}, and use the results from Section~\ref{s:prelim} to analyze the structure typical $H$-free graphs for graphs $H$ in these families. Finally, in Section~\ref{s:construction} we present constructions of infinite families of graphs with the properties specified in Section~\ref{s:applications}.

\section{Tools from~\cite{NorYud191}}\label{s:prelim}

In this section we present the results from~\cite{NorYud191}, which allow us to analyze the typical structure of typical graphs in $\Forb(H)$ for graphs $H$ constructed in the later sections.

We start by extending the definition of the witnessing partition number to general hereditary families.
The \emph{coloring number} $\chi_c(\mc{F})$ of a graph family $\mc{F}$ is the maximum integer $l$ such that $\mc{H}(s,l-s) \subseteq \mc{F}$ for some $0 \leq s \leq l$. Clearly, $\wpn(H)=\chi_c(\Forb(H))$ for every graph $H$. We say that $\mc{F}$ is \emph{thin} if $\chi_c(\mc{F}) \leq 1$.

The following is a key definition in our structural results. Let $\mc{F}$ be a hereditary graph family,  and let $l =\wpn(\mc{F})$. Let $\iota(J)$ denote the hereditary family of graphs isomorphic to induced subgraphs of a graph $J$.  We say that a graph $J$ is \emph{$\mc{F}$-reduced} if there exists an integer $0 \leq s \leq l-1$ such that 
$$\mc{P}(\iota(J),\mc{H}(s,l-1-s)) \subseteq \mc{F}.\footnote{I.e. $\mc{F}$ contains all graphs which admit a vertex partition into $l$ parts, such that the first part induces a subgraph of $J$, $s$ of the remaining parts are stable sets, and the rest are cliques.}$$ We say that $J$  is  \emph{$\mc{F}$-dangerous} if $J$ is not $\mc{F}$-reduced. Let $\red(\mc{F})$ and  $\dang(\mc{F})$ denote the families of all $\mc{F}$-reduced and $\mc{F}$-dangerous  graphs, respectively. For brevity we write  $\red(H)$ and  $\dang(H)$ instead of $\red(\Forb(H))$  and $\dang(\Forb(H))$, respectively.

Note that if $\mf{F}$ is a proper pattern such that $\mc{P}(\mf{F}) \subseteq \mc{F}$ and $|\mf{F}|=l$ then $\mc{T} \subseteq \red(\mc{F})$ for every family $\mc{T} \in \mf{F}$.  In particular, we have $\mc{P}(\mf{F}) \subseteq \mc{P}(\red(\mc{F}),l)$ for every such pattern $\mf{F}$. 
The description of typical structure of $\mc{F}$ given in \cref{t:critical} below relaxes Conjecture~\ref{c:RS} in the direction suggested by this observation:   Under several significant technical restrictions on $\mc{F}$ we show that almost all graphs in $\mc{F}$ belong to $\mc{P}(\red(\mc{F}),l)$ 

Let us now present these restrictions.
A \emph{substar} is a subgraph of a star, and an \emph{antisubstar} is a complement of a substar. We say that a hereditary family $\mc{F}$ is \emph{apex-free} if $\crit(\mc{F})$ contains a substar and an antisubstar. It turns out that assuming that the family  $\mc{F}$ is apex-free significantly simplifies analysis of its structure. We say that $\mc{F}$ is \emph{meager} if it is thin and apex-free.

We say that a hereditary family $\mc{F}$ with $l=\wpn(\mc{F}) \geq 2$ is \emph{smooth} if for every $\delta>0$ there exists $n_0$ such that $$|\mc{F}^n| \geq 2^{((l-1)/l -\delta)n}|\mc{F}^{n-1}|$$ for all  integers $n \geq n_0$.  As $|\mc{F}^n| \geq  2^{(l-1)n^2/{2l} - o(n^2)}$ for every  hereditary family as above, we expect ``reasonable" hereditary families to be smooth, yet it appears difficult to prove that a given  hereditary family is smooth without first understanding its structure.

\begin{thm}[\protect{\cite[Theorem 2.6]{NorYud191}}]\label{t:critical}
 Let $\mc{F}$ be an apex-free hereditary family, let  $l=\wpn(\mc{F}) \geq 2$, let $\mc{K} \subseteq \crit(\mc{F})$ be a finite set of graphs, and let $\mc{T} = \Forb(\mc{K})$. If $\mc{P}(\mc{T},l) \cap \mc{F}$ is smooth then almost all graphs  $G \in \mc{F}$  admit a  $(\mc{T},l)$-partition.  
\end{thm}
 
In addition to Theorem~\ref{t:critical} we will use an easy lemma which is helpful in verifying that conditions of \cref{t:critical} are satisfied. We say that a family  $\mc{F}$ is \emph{extendable} if there exists $n_0 \geq 0$ such that for every $G \in \mc{F}$ with $|V(G)| \geq n_0$ we have $G = G' \setminus v$ for some  $G' \in \mc{F}$  and $v \in V(G')$.  

 \begin{lem}[\protect{\cite[Lemma 2.9]{NorYud191}}]\label{l:cleanext}
  	Let $\mf{F}$ be a proper pattern such that every $\mc{T} \in \mf{F}$ is extendable and thin. Then  the family $\mc{P}(\mf{F})$ is smooth.
 \end{lem}
 
Our applications of Theorem~\ref{t:critical} use not only existence of a structured partition of a typical graph, but the facts that such a partition is unique and essentially balanced in the following sense. We say that a partition $\mc{X}$ of an $n$ element set is \emph{$\eps$-balanced} if $|X - n/|\mc{X}|| \leq n^{1 - \eps}$ for all $X \in \mc{X}$. 

 \begin{lem}[\protect{\cite[Corollary 2.11]{NorYud191}}]\label{l:uniqueness}  Let $\mf{F}$ be a proper pattern such that every $\mc{T} \in \mf{F}$ is  meager and extendable. Then there exists $\eps>0$ such that  almost  all graphs in $\mc{P}(\mf{F})$ admit a unique  $\mf{F}$-partition, and such a partition is $\eps$-balanced.   		
 \end{lem} 

Let $(G_1,\ldots,G_l)$ be a collection of vertex disjoint graphs, and let $X=\cup_{i \in [l]}V(G_i)$. We say that a graph $G$ is an \emph{extension} of $(G_1,\ldots,G_l)$ if $V(G)=X$, and $G_i$ is an induced subgraph of $G$ for every $i \in [l]$.
 
\begin{lem}[\protect{\cite[Lemma 2.12]{NorYud191}}]\label{l:uniqueness2} Let $\mc{T}$ be a meager hereditary family, let $l$ be an integer, and let $\eps>0$ be real. Let $(G_1,\ldots,G_l)$ be a collection of graphs such that $G_i \in \mc{T}$ for every $i \in [l]$ and $\mc{X} =(V(G_1),\ldots,V(G_l))$ 
is  an $\eps$-balanced partition of $[n]$. Then $\mc{X}$ is the unique $(\mc{T},l)$-partition of $G$ for almost every extension $G$ of $(G_1,\ldots,G_l)$. 
 \end{lem}   
 
The main application of Theorem~\ref{t:critical} in \cite{NorYud191}, which we will also need here, is a generalization of the following result of Balogh and Butterfield~\cite{BB11}. 

In~\cite{BB11} a graph $H$ is defined to be \emph{critical} if there exists an integer $n_0$ such that every $K \in \red(H)$ with $|V(K)| \geq n_0$ is either complete or edgeless. Thus a clean $H$-free profile of a graph $G$ corresponds to a partition of $V(G)$ in to cliques and stable sets (and potentially bounded size graphs), such that $H$ does not admit a partition with the same structure. 
The following characterization of critical graphs $H$ given in~\cite{BB11} shows that one indeed can find such a partition for almost all $H$-free graphs. It implies, in particular,  that critical graphs satisfy Conjecture~\ref{c:RS}.

\begin{thm}[\protect{\cite{BB11}}]\label{t:BBcritical}
	A graph $H$ is critical if and only if for almost every $G \in \Forb(H)$ we have $G \in \mc{H}(s,t)$ for some pair of non-negative integers $s$ and $t$ such that $s+t=\wpn(H)$ and   $\mc{H}(s,t) \subseteq \Forb(H)$.
\end{thm}

To describe and motivate our generalization of \cref{t:BBcritical} we need several additional definitions. We say that a set $S \subseteq V(G)$ is a \emph{core of a graph $G$} if for every $v \in V(G)$ either $v$ is adjacent to every vertex of $V(G)-S$ or  $v$ is not adjacent to any vertex in $S$. We say that a graph $G$ is an \emph{$s$-star} for an integer $s \geq 0$ if $G$ has a center of size at most $s$. Thus $0$-stars are exactly complete or edgeless graphs, and $1$-stars are induced subgraphs of stars or antistars.

We say that a hereditary family  $\mc{F}$ is \emph{$s$-critical} for an  integer $s \geq 0$ if there exists $n_0$ such that every $K \in \red(\mc{F})$ with $|V(K)| \geq n_0$ is an $s$-star. We say that a graph $H$ is $s$-critical if $\Forb(H)$ is $s$-critical. Thus $0$-critical graphs are exactly critical graphs.

Pr\"{o}mel and Steger~\cite{PS93}  have shown that $1$-critical graphs are exactly the extremal graphs according to a certain metric related to the structure of $\Forb(H)$.\footnote{The definition of critical graphs considered in~\cite{PS93} differs from our definition of $1$-critical graphs, but as we show in Section~\ref{s:PScritical} the definition we give here is equivalent.} Clearly, every $0$-critical graph is  $1$-critical, but, as noted in~\cite{BB11}, it is  not obvious whether the converse holds. Our final main result in the vein of Theorems~\ref{t:ARS} and~\ref{t:P3vague}  show that it does not.

\begin{thm}\label{t:BBPS}
	There exist infinitely many graphs which are $1$-critical, but not $0$-critical.
\end{thm}

In the proof of \cref{t:ARS} we construct a family of counterexamples to \cref{c:RS} that are $2$-critical. We suspect that the conjecture does not hold even for $1$-critical graphs. In spite of this in~\cite{NorYud191} we obtained a structural description of typical graphs in $\mc{F}$ for any $s$-critical family $\mc{F}$, which we now present.

We define an \emph{$(l,s)$-constellation} (or simply a \emph{constellation}) to be a quadruple $\mc{J}=(J,\phi,\alpha, \beta)$, where $J$ is a (possibly empty) graph, and $\phi,\alpha, \beta$ are functions such that $\phi: V(J) \to [l]$ satisfies $|\phi^{-1}(i)| \leq s$ for every $i \in [l]$,  $\alpha : V(J) \to \{0,1\}$ and $\beta :[l] \to \{0,1\}$. We say that a constellation $\mc{J}$ is \emph{irreducible} if for every $v \in V(J)$ if $\beta(\phi(v))=\alpha(v)$  then there exists $u \in V(J)-\{v\}$ such that $\phi(u)=\phi(v)$ and either $uv \in E(G)$ and $\alpha(u)=0$, or $uv \not \in E(G)$ and $\alpha(v)=1$.

Given an $(l,s)$-constellation $\mc{J}=(J,\phi,\alpha, \beta)$, define a \emph{$\mc{J}$-template} in a graph $G$ to be a tuple $(\psi,X_1,X_2, \ldots,X_l)$ such that  
\begin{itemize}
	\item $(X_1,X_2, \ldots,X_l)$ is a partition of $V(G)$,
	\item $\psi:V(J) \to V(G)$ is an embedding satisfying $\psi(v) \in X_{\phi(v)}$ for every $v \in V(J)$,
\end{itemize}
and, denoting the image of $\psi$ by $Z$, we have
\begin{itemize}	
	\item if $\alpha(v)=1$ then $\psi(v)$ is adjacent to every vertex in $X_{\phi(v)} - Z$  in $G$, and, otherwise,
	$\psi(v)$ is adjacent to no vertex in $X_{\phi(v)} - Z$, and,
	\item if $\beta(i)=1$ then $X_i - Z$ is a clique in $G$, and, otherwise,  $X_i -Z$ is an independent set.
\end{itemize}
Thus, in particular, $Z \cap X_i$ is a core of $G[X_i]$ for every $i \in l$ and thus  $X_i$ induces an $s$-star in $G$ for every $i \in [l]$ and if $|X_i-Z|\geq 2$ and $\mc{J}$ is irreducible then  $Z \cap X_i$ is a minimal core of $G[X_i]$.

Let $\mc{P}(\mc{J})$ denote the family of induced subgraphs of all graphs which admit a 
$\mc{J}$-template. 

\begin{thm}[\protect{\cite[Theorem 2.17]{NorYud191}}]\label{t:critical2}
	Let $\mc{F}$ be an $s$-critical hereditary family with $\chi_c(\mc{F})=l$. Then for almost every graph in $G \in \mc{F}$ there exists an irreducible $(l,s)$-constellation $\mc{J}$ such that $G \in \mc{P}(\mc{J}) \subseteq  \mc{F}$
\end{thm}

Note that if $\mc{J}$ is an $(l,0)$-constellation then a $\mc{J}$-template in a graph $G$ is a partition of $V(G)$ into $l$ homogenous sets, the fixed number of which are cliques. Thus \cref{t:critical2} does indeed generalize one of the directions of \cref{t:BBcritical}. 

A $\mc{J}$-template is similar to the structure proposed by \cref{c:RS}, but in addition to prescribing the structure on the parts of the partition given by the template, we prescribe the behavior of a finite number of additional edges. As \cref{t:ARS} shows this additional restriction is sometimes necessary.  It is tempting to attempt to formulate a common generalization of  \cref{c:RS}  and  \cref{t:critical2}, but we were unable to find a plausible one.

Finally, we need a bound on the number of graphs admitting an $\mc{J}$ template.

\begin{lem}[\protect{\cite[Lemma 2.16]{NorYud191}}]\label{t:speed}
	Let  $\mc{J}=(J,\phi,\alpha,\beta)$ be an irreducible $(l,s)$-constellation then 
	$$ |\mc{P}^n(\mc{J})|= \Theta(n^{|V(J)|}|\mc{H}^n(l,0)|).$$
\end{lem}

\section{Exotic examples}\label{s:applications}

\subsection{A family of counterexamples to the Reed-Scott's conjecture}\label{s:ARS}

In this subsection we explicitly define a family  of graphs, which satisfy Theorem~\ref{t:ARS}. 

First, an extra notation. Given two families of graphs $\mc{F}_1$ and $\mc{F}_2$, let $\mc{F}_1 \vee \mc{F}_2$ denote the family of all graphs  which are disjoint unions of a graph in $\mc{F}_1$ and a graph in 
$\mc{F}_2$. Similarly, let $\mc{F}_1 \wedge \mc{F}_2$ denote the family of joins of graphs in $\mc{F}_1$ with graphs in $\mc{F}_2$.\footnote{For example, $\mc{S} \vee \mc{S} = \mc{S}$ and $\mc{S} \wedge \mc{S}$ is the family of complete bipartite graphs.} Clearly, if $\mc{F}_1$ and $\mc{F}_2$ are hereditary then so are $\mc{F}_1 \vee \mc{F}_2$  and $\mc{F}_1 \wedge \mc{F}_2$. 

Given a positive integer $l$, we say that a graph $H$ is an \emph{$l$-ARS-graph} or simply   \emph{an ARS-graph} if $H$ satisfies the following conditions:
\begin{description}
	\item[(ARS1)] For every $1 \leq s \leq l$,  $V(H)$ can be partitioned into $s$ stable sets and $l-s$ cliques;
	\item[(ARS2)] For each graph class $$\mc{G} \in \{\iota(K_1) \vee \mc{C}, \iota(S_3) \wedge \mc{C}, \iota(C_4) \wedge \mc{C},  \iota(\bar{P}_3) \wedge \mc{C}\},$$  
	$V(H)$ can be partitioned into $l-1$ cliques and a set inducing a graph in $\mc{G}$;
	\item[(ARS3)] There exists a partition $\mc{X}_0=\{X_1,X_2,\ldots,X_{l}\}$ of $V(H)$, such that $X_1,X_2,\ldots,X_{l-2}$ are cliques, each of $X_{l-1}$ and $X_{l}$  induce a  subgraph of $H$ with exactly one non-edge, and the vertices of these two non-edges form an independent set in $H$;
	\item[(ARS4)] For every partition $\mc{X} \neq \mc{X}_0$ of $V(H)$ with $|\mc{X}|=l$ there exists $X \in \mc{X }$  such that $H[X]$ contains at least two non-edges. 
\end{description}

In Section~\ref{s:ARSexist} we prove the following.

\begin{thm}\label{t:ARSexist} For infinitely many integers $l$ there exists a $l$-ARS-graph.
\end{thm}

Meanwhile, we will  show that every $l$-ARS-graph satisfies Theorem~\ref{t:ARS}, thus proving Theorem~\ref{t:ARS} modulo Theorem~\ref{t:ARSexist}.

We start with a few of easy lemmas. 

\begin{lem}\label{l:2nonedges} For every positive integer $h$ there exists $N>0$ satisfying the following. Let  $G$ be a graph with $|V(G)| \geq N$ and at least two non-edges then $G$ contains an induced subgraph $J$ with $|V(J)|=h$ such that $J$ is either edgeless, or an antistar, or a join of one of the graphs in  $\{S_3, C_4, \bar{P}_3 \}$ with a complete graph. 
\end{lem}

\begin{proof}
	Let $n$ be a positive integer such that every graph on $n$ vertices contains a homogenous set on $h$ vertices. We show that $N=5n$ satisfies the lemma.
	
	Let $G$ be as in the lemma statement. We suppose that $G$ contains no stable set on $h$ vertices, as otherwise the lemma holds. Thus every set of $n$ vertices of $G$ contains a clique on $h$ vertices. Suppose first that there exists $v \in V(G)$ with at least $n$ non-neighbors, then $v$ together with a clique of size $h-1$ chosen among its non neighbors induces a desired antistar. Thus we assume that every vertex of $G$ has at most $n$ non-neighbors. It follows that every set of at most four vertices of $G$ has at least $n$ common neighbors, and so there exists a clique of size $h$ among those neighbors.
	
	Thus it suffices to show that if $G$ contains at least two non-edges, then it contains an induced subgraph isomorphic to one of $S_3, C_4$ or $\bar{P}_3$, but this is clear. 
\end{proof}

\begin{cor}\label{c:ARSreduced} 
	 Let $H$ be an $l$-ARS graph then   $\wpn(H)=l$, and there exists an integer $n_0$ such that 
	every graph $G \in \red(H)$ with $|V(G)| \geq n_0$ contains at most one non-edge.  In particular, $H$ is $2$-critical.
\end{cor}

\begin{proof}
		It follows from (ARS1) and (ARS2) that $\wpn(H) < l+1$,	and it  follows from (ARS3) and (ARS4) that $V(H)$ can not be partitioned into $l$ cliques, implying that $\wpn(H) \geq l$. Thus $\chi_c(H)=l$.
		
		Let $h=|V(H)|$ and let $n_0$ be such that the conclusion of \cref{l:2nonedges} holds with $N=n_0$.
		Suppose for a contradiction that there exists $G \in \red(H)$ with $|V(G)| \geq n_0$ such that $G$ has at least two non-edges. Then by \cref{l:2nonedges} there exists $J \in \red(H)$ such that $|V(J)| \geq |V(H)|$ and $J$ is either edgeless, or an antistar, or a join of one of the graphs in  $\{S_3, C_4, \bar{P}_3 \}$ with a complete graph. It follows from (ARS1) and (ARS2) that for every $0 \leq s \leq l-1$ we can partition $V(H)$ into $s$ stable sets, $l-1-s$ cliques and a subgraph of $J$, a contradiction. 
\end{proof}	

Let $\mc{I}=\iota(S_2) \wedge \mc{C}$ denote the family of graphs with at most one non-edge, and  let $\mc{A}(l)$ denote the family of all graphs $G$ such that there exists an $(\mc{I},l)$-partition $\mc{X}$ of $G$ such that $X \cup X'$ does not contain an independent set of size four for all  $X,X' \in \mc{X}$.  The following theorem describes the structure of typical $H$-free graphs for an $l$-ARS graph $H$.

\begin{thm}\label{t:ARSstructure} Let $H$ be an $l$-ARS graph. Then \begin{itemize}
	\item[(i)]	$\mc{A}(l) \subseteq \Forb(H)$,  
	\item[(ii)] almost all $H$-free graphs are in $\mc{A}(l)$,
	\item[(iii)] $$|\Forb^n(H)| = \Theta(n^{2l}|\mc{H}^n(l,0)|).$$ 
	\end{itemize}
\end{thm}	

\begin{proof}
The condition (i) follows from (ARS3) and (ARS4). 

Let $\mc{J}=(J,\phi,\alpha,\beta)$ be an $(l,2)$-constellation  such that $\mc{P}(\mc{J}) \subseteq \Forb(H)$. We claim that $\mc{P}(\mc{J}) \subseteq A(l)$. By \cref{c:ARSreduced} and \cref{t:critical2} this claim implies (ii). By  \cref{t:speed} it additionally implies $|\Forb^n(H)| = O(n^{2l}|\mc{H}^n(l,0)|).$ 
 
Note that (ARS1)-(ARS3) imply that $\beta$ and $\alpha$ are identically one, and there does not exist an independent set of $\{u_1,v_1,u_2,v_2\}$ in $J$ such that $\phi(u_i)=\phi(v_i)$ for $i=1,2$. This observation immediately implies  $\mc{P}(\mc{J}) \subseteq A(l)$, as claimed.

Let $J$ be a graph with $V(J)=\{u_i,v_i\}_{i=1}^l$ obtained from a complete graph by deleting edges $u_iv_i$ for every $i \in [l]$. Let $\phi(u_i)=\phi(v_i)=i$ for every $i \in l$, and let  $\beta$ and $\alpha$ be identically one, and let $\mc{J}=(J,\phi,\alpha,\beta)$. Then $\mc{P}(\mc{J})$ consists of all graphs $G$ such that there exists an $(\mc{I},l)$-partition $\mc{X}$ of $G$ so that if $u_1,v_1 \in X_1$,  $u_2,v_2 \in X_2$ are pairwise distinct vertices for some $X_1,  X_2 \in \mc{X}$ and $u_iv_i \not \in E(G)$ for $i=1,2$ then $u_1u_2,u_1v_2,v_1u_2,v_1v_2 \in E(G)$. It follows that $\mc{P}(\mc{J}) \subseteq A(l)$. As $\mc{J}$ is  irreducible, we have $|\mc{P}^n(\mc{J})| = \Theta(n^{2l}|\mc{H}^n(l,0)|)$ by  \cref{t:speed}. Together with the upper bound established above this implies (iii)
\end{proof}

\begin{proof}[Proof of \cref{t:ARS}]
	We show that if $H$ is an $l$-ARS graph for some $l \geq 2$,  then almost all $H$-free graphs admit no clean $H$-profile,
	
	Let $\mf{F}$ be a clean $H$-free pattern. Then $\mc{F} \subseteq \red(H)$ for every $\mc{F} \in \mf{F}$. By \cref{c:ARSreduced} and \cref{l:uniqueness}, there exists $\eps>0$ such that almost every graph in $G \in \mc{P}(\mf{F})$ admits an $(\mc{I},l)$-partition and such a partition is unique and $\eps$-balanced. By (ARS3) at most one element of $\mf{F}$ contains $\mc{I}$, and thus if $V(G)$ is sufficiently large the above partition must be a partition of $V(G)$ into a set inducing at most one non-edge and $l-1$ cliques, corresponding to a  $\mc{J}$ template for an $(l,2)$-constellation $\mc{J}=(J,\phi,\alpha,\beta)$ with $|V(J)|=2$. Thus by  \cref{t:speed} and \cref{t:ARSstructure} (iii) we have $$|\mc{P}^n(\mf{F})|=O(n^{2}|H^n(l,0)|) = o(|\Forb^n(H)|),$$
	as desired.
\end{proof}	

\subsection{Graphs with no asymptotic linear Erd{\H{o}}s-Hajnal property.}\label{s:noEH}

In this section we describe a family of graphs satisfying Theorem~\ref{t:P3vague}.

Let $l$ be an integer. 
We say that a graph $H$ is a \emph{$(P_3,l)$-jumble} or simply a \emph{$P_3$-jumble}  if \begin{itemize}
	\item[(S1)] for every $0 \leq s \leq l-1$,  $V(H)$ can be partitioned into $s$ stable sets, $l-1-s$ cliques, and a set $Z$ such that $H[Z]$ is isomorphic to $P_3$, and
	\item[(S2)] for every partition $X_1,X_2,\ldots,X_{l}$ of $V(H)$ there exists $i$ such that $H[X_i]$ is not $P_3$-free.
\end{itemize}	

In \cref{s:P3construction} we will show the following.

\begin{thm}\label{t:existence}There exist  $(P_3, l)$-jumbles for infinitely many postive integers $l$.
\end{thm}

In this section we show that every $P_3$-jumble has no asymptotic linear Erd\H{o}s-Hajnal property, thus proving \cref{t:P3vague} modulo Theorem~\ref{t:existence}. The main step of the argument is the description of a typical $H$-free graph for a $P_3$-jumble $H$, which follows directly from Theorem~\ref{t:critical}.

\begin{thm}\label{t:P3structure}
	Let $H$ be an $(P_3,l)$-jumble, and let $\mc{T}=\Forb(P_3)$. Then $\mc{P}(\mc{T},l) \subseteq \Forb(H)$. Conversely, there exists $\eps>0$ such that almost every graph $G \in \Forb(H)$ admits a unique $\eps$-balanced $(\mc{T},l)$-partition.
\end{thm}

\begin{proof} Let $\mc{F} = \Forb(H)$.
	We have $\mc{P}(\mc{T},l) \subseteq \mc{F}$ by (S2). Thus, $\wpn(\mc{F}) \geq l$. On the other hand (S1) implies that $\wpn(\mc{F}) \leq l$, and $P_3$ is dangerous for $\mc{F}$. As $\mc{T}$ is clearly extendable, \cref{l:cleanext} implies that $\mc{P}(\mc{T},l)$ is smooth. Applying Theorem~\ref{t:critical} and Lemma \ref{l:uniqueness} with $\mc{K}=\{P_3\}$ now yields the conclusion.
\end{proof}

\begin{cor}\label{c:P3concrete}
	Let $H$ be a $P_3$-jumble. Then almost all $H$-free graphs $G$ on $n$ vertices satisfy $h(G)=O\left( \frac{n}{ \log n}\right)$.  
\end{cor}

\begin{proof} Let $l=\wpn(H)$, $\mc{F} = \Forb(H)$, $\mc{T}=\Forb(P_3)$ and let $\eps > 0$ be as in \cref{t:P3structure}. 
We say that $(G,\mc{X})$ is a \emph{good} pair if $G \in \mc{F}$, and $\mc{X}$ is the unique $\eps$-balanced $(\mc{T},l)$-partition of $G$.
By \cref{t:P3structure} almost every $G \in \mc{F}$ belongs to a (unique) good pair. 

We claim that for every $\eps$-balanced partition $\mc{X}$ of $[n]$, we have $h(G) = O \left( \frac{n}{ \log n}\right)$ for almost every good pair $(G,\mc{X})$. Clearly this claim implies the corollary. 

Let $\mc{X} = (X_1,\ldots,X_l)$, and let $\mc{Z}$ be the set of all possible sequences $(G_1,\ldots,G_l)$ such that $V(G_i)=X_i$ and $G_i \in \mc{T}$. Every such sequence extends to $2^m$ graphs $\mc{F}^n$, where $m=\sum_{1 \leq i <j \leq l}|X_i||X_j|$. Let $\mc{Z}(C)$ be the set of all sequences in $\mc{Z}$ such that $h(G_i) \leq C \frac{n}{ \log n}$ for every  $i \in [l]$. By \cref{o:P3homogenous} there exists $C>0$ independent on $n$ such that $|\mc{Z}(C)| = |\mc{Z}|-o(\mc{Z})$. By \cref{l:uniqueness2} almost every extension of a given sequence in $\mc{Z}(C)$ gives rise to a good pair, implying that at least $(1-o(1))2^m|\mc{Z}|$ good pairs of the form $(G,\mc{X})$ satisfy $h(G) \leq lC \frac{n}{ \log n}$. On the other hand the remaining sequences in $\mc{Z}$ correspond to $o(2^m |\mc{Z}|)$ good pairs, which finishes the proof of the claim.
\end{proof}	
 
\subsection{Exotic PS-critical graphs}\label{s:PScritical} 

In this section we present a family of graphs satisfying \cref{t:BBPS}. The description of these graphs is very similar to the description of ARS-graphs in Section~\ref{s:ARS} and some of the results from that section carry over after minor modifications.   

First, let us show that the definition of $1$-critical graphs  is equivalent to the original definition given in~\cite{PS93}, as promised in the introduction.  We need the following analogue of Lemma~\ref{l:2nonedges}.

\begin{lem}\label{l:nonclique} For every positive integer $h$ there exists $N$ satisfying the following. If $G$  is a graph with $|V(G)| \geq N$ and $G$ is neither complete nor edgeless, then $G$ contains an induced subgraph $J$ such that $|V(J)|=h$ and either $J$ or $\bar{J}$ is a star or has exactly one edge.
\end{lem}	

\begin{proof}
	Let $N$ be such that every graph on at least $N$ vertices contains a homogenous set of size at least $2h$. Thus without loss of generality we assume that $G$ contains a maximal independent set $S$ such that $|S| \geq 2h$. As $G$ is not edgeless there exists $v \in V(G) - S$ and $v$ has a neighbor in $S$. If $v$ has at least $h$ neighbors in $S$ then $G[S \cup \{v\}]$ contains a star on $h$ vertices and, otherwise, $G[S \cup \{v\}]$ contains an $h$ vertex induced subgraph with exactly one edge. 
\end{proof} 

\begin{lem}\label{l:nonstar} For every positive integer $h$ there exists $N$ satisfying the following. If $G$  is a graph with $|V(G)| \geq N$ such that neither $G$ nor $\bar{G}$ is edgeless or a star, then $G$ contains an induced subgraph $|V(J)|=h$ and either $J$ or $\bar{J}$ has exactly one edge,  or is obtained from a star by deleting one edge, or is join of a graph on two vertices and an independent set. 
\end{lem}	

\begin{proof} By \cref{l:nonclique} there exists $N$ such that every graph $G$ with $|V(G)| \geq N$ such that neither $G$ nor $\bar{G}$ is edgeless contains  an induced subgraph $J'$ such that $|V(J')|\geq 2h$ and either $J'$ or $\bar{J}'$ is a star or has exactly one edge.  We may assume without loss of generality that $J'$ is a maximal induced subgraph of $G$ which is a star. Let $S$ be the set of all the leaves of $J'$, and let $u$ be the center of $J'$. 
	
Consider arbitrary $v \in V(G) - V(J')$. The join of $G[\{u,v\}]$ and the neighborhood of $v$ in $S$ is an induced subgraph of $G$. Therefore, if $v$ has at least $h$ neighbors in $S$ then the lemma holds, and so we assume that $v$  has at least $h$ non-neighbors in $S$. If $v$ has at least one neighbor in $S$ then we can find an $h$ vertex induced subgraph of $G$ with exactly one edge. It remains to consider the case, when $\{v\} \cup S$ is independent. By maximality of $J$, we have $uv \not \in E(G)$, and so 
$V(J) \cup \{v\}$ induces a  graph obtained from a star by deleting one edge, as desired.
\end{proof}

In~\cite{PS93, BB11} a graph $H$ is defined to be \emph{PS-critical} if every sufficiently large join of a graph on two vertices and an independent set is dangerous for $\Forb(H)$, and the same is true for every sufficiently large graph obtained from a star by deleting an edge, as well as for the complements of the above graphs.  The next corollary, which is an immediate consequence of \cref{l:nonstar}, implies that this definition coincides with the definition of $1$-critical given in the introduction.

\begin{cor}\label{c:PSisPS} 
Let $\mc{F}$ be a hereditary family. Suppose that there exists graphs $J_1, J_2, J_3$ such that $J_i,\bar{J}_i \in \dang(\mc{F})$ for every $i \in [3]$, $J_1$ is obtained from a star by deleting an edge, and $J_2$ and $J_3$ are joins of an independent set with $K_2$ and $S_2$, respectively. Then 
$\mc{F}$ is $1$-critical.
\end{cor}

\begin{proof}
It follows from  \cref{l:nonstar} that
there exists an integer $n_0$ such that for every graph  $K \in \red(\mc{F})$ with  $|V(K)| \geq n_0$ either $K$ or $\bar{K}$ is edgeless or a star. 	
\end{proof}	

We are almost ready to define the family of graphs satisfying Theorem~\ref{t:BBPS}. In addition to the classes of graphs used  to define ARS-graphs, we need to introduce notation the following  graph class. We denote by $\mc{C}^{+}$ the family of all graphs $G$ such that either $G$ is complete, or $G \setminus v$ is complete for some vertex $v \in V(G)$ such that $\deg(v) \leq 1$.

Let $l$ be a positive integer.
We say that a graph $H$ is an \emph{$l$-EPS-graph} (or simply an \emph{EPS-graph}) if $H$ satisfies the following conditions:
\begin{description}
	\item[(EPS1)] For every $1 \leq s \leq l$,  $V(H)$ can be partitioned into $s$ stable sets and $l-s$ cliques;
	\item[(EPS2)] For each graph class $$\mc{G} \in \{\iota(S_2) \vee \mc{C}, \iota(K_2) \vee \mc{C}, \iota(S_2) \wedge \mc{C}, \mc{C}^{+}\},$$  
	$V(H)$ can be partitioned into $l-1$ cliques and a set inducing a graph in $\mc{G}$;
	\item[(EPS3)] There exists no partition $\mc{X}=(X_1,X_2,\ldots,X_l)$ of $V(H)$ such that $H[X_1]$ is a clique or a complement of a star, and $X_i$ is a clique in $H$ for $2 \leq i \leq l$.
\end{description}

\begin{lem}\label{l:EPS}
	Every EPS-graph is $1$-critical, but not $0$-critical.
\end{lem}

\begin{proof}
	Let $H$ be an $l$-EPS-graph. It follows from (EPS1) and (EPS2) that $\wpn(H) \leq l$, and it follows from (EPS3) that $\wpn(H) \geq l$. Therefore $\wpn(H) = l$, and (EPS1) and (EPS2) further  guarantee that there exists graphs $J_1,J_2$ and $J_3$ satisfying the conditions of \cref{c:PSisPS} for $\mc{F} = \Forb(H)$. Thus $H$ is $1$-critical by \cref{c:PSisPS}. 
	
	By (EPS3) there exists no partition  $(P_1,P_2,\ldots,P_l)$ of $V(H)$ such that $H[P_1]$ is a complement of a star, and $P_2,\ldots,P_l$ are cliques in $H$. Therefore $H$ is not $0$-critical, as no complement of a star is $\Forb(H)$-dangerous.
\end{proof}

By \cref{l:EPS} the following theorem, which is proved in Section~\ref{s:ARSexist}, implies \cref{t:BBPS}.

\begin{thm}\label{t:EPSexist}
	There exists infinitely many EPS-graphs.
\end{thm}

\section{Constructions}\label{s:construction}

\subsection{Proof of Theorems~\ref{t:ARSexist}and~\ref{t:EPSexist}}\label{s:ARSexist}

In this section we construct infinite families of ARS-graphs and of EPS-graphs. 

Let us start by informally sketching our construction.
We start by specifying the partitions of $V(H)$ which will satisfy properties (ARS1)-(ARS3) and (EPS1)-(EPS2), respectively. These partitions will be chosen using a randomized procedure subject to certain transversality conditions. The graph $H$ will then be constructed by specifying its structure within each part of the partitions. The transversality and randomness will be used to ensure that these specifications don't conflict with each other, and that every large enough ``structured'' subgraph
of $H$ is close to being a part of one of the partitions. The last condition will guarantee that the respective conditions (ARS4) and (EPS3) also holds.

We use the following definitions in the description of the properties of the constructed partitions. 
We say that families of finite sets $\mc{P}$ and $\mc{P}'$ are \emph{transversal} if $|P \cap P'| \leq 1$ for all $P \in \mc{P}'$ and $P' \in \mc{P}'$.  We say that a set $S$ is \emph{covered by  $\mc{P}$}  if  there exists $P \in \mc{P}$ such that $S \subseteq P$, and otherwise,  we say that a set $S$ is \emph{uncovered by  $\mc{P}$}.  We say that an element of $x \in P$ is \emph{$P$-exclusive with respect to $\mc{P}$} for $P \in \mc{P}$ if for every $P' \in \mc{P}$ such that $P' \neq P$ and $x \in P'$, we have $|P \cap P'|=1$. We say that $P \in \mc{P}$ is \emph{distinctive with respect to $\mc{P}$} if there exist at least two $P$-exclusive elements. We say that a set $S$ is \emph{$\mc{P}$-wild} if at least two distinct two element subsets of  $S$  are uncovered by $\mc{P}$, and, otherwise we say that
$S$ is \emph{$\mc{P}$-tame}. We say that a family of sets is \emph{$\mc{P}$-tame} if every element of it is $\mc{P}$-tame.

The technical part of the proof of Theorems~\ref{t:ARSexist} and~\ref{t:EPSexist} consists of establishing the following lemma.

\begin{lem}\label{l:partitions}
Let $l$ and $k$ be positive integers such that $l \geq 500 k$, $k^{3/2} \geq 3600l$ and $(l-k)k$ is divisible by $l$. Let $X$ be a set with $|X|=(l-k)k+l$. Then there exist partitions $\mc{Q}=(Q_0,Q_1,Q_2,\ldots,Q_k),\mc{P}_0,\ldots,\mc{P}_4$ of $X$ and a partition $\mc{R}$ of $X-Q_0$ with the following properties:
\begin{description}
	\item[(Q)]  $|Q_0|=l$ and $|Q_j|=l-k$ for every $j \in [k]$,
	\item[(R)]  $|\mc{R}|=l-k$, and $\mc{R}$ and $\mc{Q}$ are transversal,
\end{description}	
 moreover, for every $i \in \{0,1,2,3,4\}$ we have
\begin{description}
	\item[(P1)] $|\mc{P}_i|=l$ and $|P|=k(l-k)/l+1$ for every $P \in \mc{P}_i$, 
	\item[(P2)] $\mc{P}_i$ and $\mc{Q}$ are transversal,
\end{description}	
finally, let $\mc{P}=\mc{R} \bigcup \left(\cup_{j=0}^4\mc{P}_j \right)$, then 
\begin{description}
	\item[(P3)] at least two sets of $\mc{P}_i$ are distinctive with respect to $\mc{P}$,
	\item[(P4)] let $Z \subseteq X$ be such that $|Z| \leq 9$, then for every   $\mc{P}$-tame partition $\mc{P}_*$ of $X - Z$ with $|\mc{P}_*|=l$ there exists $0 \leq i \leq 4$ such that $\mc{P}_*$ coincides with $\mc{P}_i$ on $X - Z$.
\end{description}	
\end{lem}

Before proceding to the proof of \cref{l:partitions} we derive Theorems~\ref{t:ARSexist} and~\ref{t:EPSexist} from it.

\begin{proof}[Proof of Theorem~\ref{t:ARSexist}.]
	Given  positive integers $l,k$  satisfying the conditions of \cref{l:partitions} we will construct a $l$-ARS-graph $H$.   Let $X,\mc{Q},\mc{R},\mc{P}_0,\ldots,\mc{P}_4,\mc{P}$ be as in  \cref{l:partitions}. 
	
	Let $V(H)=X$. The edges of $H$ are determined as follows. 	
Let $P'_0,P''_0$ be two  sets in $\mc{P}_0$ distinctive with respect to $\mc{P}$, and let $P_i$ be a distinctive set in  $\mc{P}_i$ for $1 \leq i \leq 4$. Let $\mc{Z}=\{P'_0,P''_0,P_1,\ldots,P_4\}$. For each $P \in \mc{P}$ we join every pair of vertices in $P$ by an edge, and we modify $H[Z]$ for $Z \in \mc{Z}$ as follows. We delete an edge from each of $P'_0$ and $P''_0$, delete all edges incident to a single vertex in $P_1$, two edges sharing an end in $P_2$, a  matching of size two in $P_3$, and edges of a triangle in $P_4$. As the sets in $\mc{Z}$ are distinctive, we can do the deletions so that all the deleted edges are incident to the vertices in the corresponding sets which are exclusive with respect to $\mc{P}$, and so do not belong to any other set in $\mc{P}$. This finishes the description of the construction of $H$.

It is not hard now to verify that the properties (ARS1)-(ARS4) holds. Note that $\mc{Q}$ is a partition of $V(H)$ into  $k+1$ stable sets, and $\mc{R} \cup \{Q_0\}$ is a partition of $V(H)$ into  $l-k$ cliques and a stable set. Thus (ARS1) holds. Partitions  $\mc{P}_1,\ldots,\mc{P}_4$ satisfy (ARS2), and $\mc{P}_0$ satisfies (ARS3) by construction. Finally, suppose for a contradiction that (ARS4) does not hold. Thus there exists a partition $\mc{X}$ of $X$ with $|\mc{X}|=l$ such that $\mc{X} \neq \mc{P}_0$ and every part of $\mc{X}$ induces in $H$ a subgraph with at most one non-edge. Thus by construction every element of $\mc{X}$ is $\mc{P}$-tame. It follows from (P4) that $\mc{X}=\mc{P}_i$ for some $1\leq i \leq 4$, and thus some part of $X$ induces at least two non-edges, a contradiction.
\end{proof}

\begin{proof}[Proof of Theorem~\ref{t:EPSexist}.]
	The proof is analogous to the proof of Theorem~\ref{t:ARSexist} above, the only difference is that we adapt the description of the edges of $H$ to satisfy (EPS2)-(EPS3),  rather than (ARS2)-(ARS4), as follows.
	
	As in the proof of Theorem~\ref{t:ARSexist}, let $P_i$ be a distinctive set in  $\mc{P}_i$ for $1 \leq i \leq 4$ and let $v_i,u_i \in P_i$ be $P_i$-exclusive with respect to $\mc{P}$. Let $Z'=\{u_1,v_1,\ldots,u_4,v_4\}$.
	 For each $P \in \mc{P} -\mc{P}_0$ we join every pair of vertices in $P$ by an edge, and we modify $H[P_i]$ for $1 \leq i \leq 4$. 
	
	We delete the following edges: $u_1v_1$, all edges in $H[P_2]$ incident to $u_2$ or $v_2$,  all edges in $H[P_3]$ incident to $u_3$ or $v_3$, except $u_3v_3$, and all edges  in $H[P_4]$ incident $u_4$, except $u_4v_4$. As before  restricting deletions to the edges incident to the exclusive vertices ensures that these deletions do not affect the subgraphs induced by other sets in $\mc{P}$. Finally, for each $z \in Z'$ we add to $H$ edges joining $z$ to all vertices $x \in X$ such that $\{x,z\}$ is uncovered by $\mc{P}$  and $x \in Q_i$ for some  $1 \leq i \leq k/2$.   
	
	As in the proof  of Theorem~\ref{t:ARSexist}, note that partitions $\mc{Q}$ and $\mc{R} \cup \{Q_0\}$ satisfy (EPS1), and partitions   $\mc{P}_1,\ldots,\mc{P}_4$ satisfy (EPS2), by construction. 
	
	It remains to verify (EPS3). Suppose that $\mc{X}=(X_1,\ldots,X_l)$ is a partition of $X$ violating (EPS3). Let $x_1 \in X_1$ be such that $X_1 - \{x_1\}$ is a clique in $H$, and let  $Z = Z' \cup \{x_1\}$. Then the restriction $\mc{P}_*$ of $\mc{X}$ to $X - Z$ is as in \cref{l:partitions} (P4). Thus $\mc{P}_*$ coincides with $\mc{P}_i$ on $X-Z$ for some $i \in [4]$. However, $P_i$ is not a subset  of a part of $\mc{X}$. Thus there exists $v \in P_i$, $P' \in \mc{P}_i$, $P' \neq P_i$  and $j \in [l]$ such that  $\{v\} \cup (P' - Z) \subseteq X_j$. But $v$ is neither complete nor anticomplete to $P' - Z$ by construction, a contradiction. Thus (EPS3) holds. 
\end{proof}

 \begin{proof}[Proof of \cref{l:partitions}]
 	Clearly, it is possible to select $\mc{Q}$ and $\mc{R}$ as in the lemma statement satisfying (Q) and (R).

Given $\mc{Q}$ and $\mc{R}$, we select $\mc{P}_0,\ldots,\mc{P}_4$  from the set of all partitions of $X$ satisfying (P1) and (P2) uniformly and independently at random. We will  show that with positive probability both (P3) and (P4) hold. Rather than verifying (P4) directly, first, for every pair $\{\mc{P}',\mc{P}''\} \subset \{\mc{R}, \mc{P}_0, \ldots ,\mc{P}_4\}$, we will require that 
\begin{itemize}
		\item[(P5)] $|P' \cap P''| \leq \sqrt{k}$ for all $P' \in \mc{P}'$ and $P'' \in \mc{P}''$,  
		\item[(P6)] For all $\mc{P}'_* \subseteq \mc{\mc{P}'}$ and $\mc{P}''_* \subseteq\mc{P}''$ such that $|\mc{P}'_*| \geq l/10$ and $|\mc{P}''_*| \geq \sqrt{k}/10$ we have $$(\cup_{P' \in \mc{P}'_*}P')  \cap (\cup_{P'' \in \mc{P}''_*}P'') \neq \emptyset.$$
\end{itemize}

Our first goal is to show that each of (P3),(P5) and (P6) is satisfied with probability at least $3/4$.

We start with (P3). We say that $x \in X$ is \emph{$\mc{P}_0$-exclusive} if $x$ is $P$-exclusive with respect to $\mc{P}$ for $P \in \mc{P}_0$ such that $x \in P$.
 Let $\mc{P}'=\mc{R} \cup  \mc{P}_1 \ldots \cup \mc{P}_4$. Fix arbitrary $x \in X$ and let $y(i)$ be the unique element of $Q_i$ such that $\{x,y(i)\}$ is covered by $\mc{P}_0$, if such an element exists. It is easy to see that the probability that  $\{x,y(i)\}$ is covered by $\mc{P}'$ is at most $5/(l-k)$. It follows that the probability that $x$ is not $\mc{P}_0$-exclusive is at most $5k/(l-k)$. Therefore, the expected number of not $\mc{P}_0$-exclusive elements is at most $$|X|\frac{5k}{l-k} = 5k^2 + \frac{5kl}{l-k} \leq 10k^2,$$ and so the probability that there are at least $200k^2$ elements of $X$ which are not $\mc{P}_0$-exclusive is at most $1/20$. It follows that (P3) holds for $\mc{P}_0$ with probability at least $19/20$. By symmetry, (P3) holds for all $i$ with probability at least $3/4$.
 
Moving on to (P5), we say that $P$ is a \emph{full $\mc{Q}$-transversal} if $|P \cap Q| = 1$ for all $Q \in \mc{Q}$. Let $P_1,P_2$ be two full  $\mc{Q}$-transversal chosen uniformly and independently at random. Then $$\Pr[|P_1 \cap P_2| \geq \sqrt{k}] \leq \frac{(k+1)^{\sqrt{k}}}{(l-k)^{\sqrt{k}}} \leq \frac{1}{2^{\sqrt{k}}} < \frac{1}{60l^2}.
$$
Note that $P'$, $P''$ and in (P5) can be considered as subsets of two  full  $\mc{Q}$-transversal chosen uniformly and independently at random. Thus the inequality above and the union bound imply that  (P5) fails with probability at most $1/4$.

Next, for (P6), let  $Y=\cup_{P' \in \mc{P}'_*}P' $, and let  $Z=\cup_{P'' \in \mc{P}''_*}P'' $. Let $Y_i=Q_i \cap Y$ and let $Z_i = Q_i \cap Z$ for every $i \in [k]$. Note that $|Y_i| \geq l/10-k \geq l/20$ for every $i$, as at most $k$ sets in $\mc{P}'$ are disjoint from $Y_i$. 

Let $$B=\{(i,P'') | i \in [k], P'' \in  \mc{P}''_*, Q_i \cap Z = \emptyset\}.$$
As every $P'' \in \mc{P}''$ is disjoint from at most $k - k(l-k)/l$ sets in $\mc{Q}$, we have $$|B| \leq \frac{k^2| \mc{P}''_*|}{l}.$$
Let $J=\{i \in [k] \: | \: |Z_i| \geq \sqrt{k}/30\}$. As every $i \in [k]-J$ belongs to at least $| \mc{P}''_*| - \sqrt{k}/30 \geq 2| \mc{P}''_*|/3$ elements of $B$ we have
$$(k-|J|)\frac{2| \mc{P}''_*|}{3} \leq k^2| \mc{P}''_*|/l,$$
and so $|J| \geq k/2$. We have $$\Pr[Y_i \cap Z_i =\emptyset] \leq \left(1 - \frac{|Y_i|}{l-k} \right)^{|Z_i|} \leq  \left(\frac{19}{20} \right)^{\sqrt{k}/30} \leq \left(\frac{1}{2}\right)^{\sqrt{k}/600}$$ for every $i \in J$. The corresponding events are independent for all $i \in J$, and so we have $\Pr[Y \cap Z =\emptyset] \leq 1/2^{k^{3/2}/1200}.$ Summing over all possible choices of subset $\{\mc{P}',\mc{P}''\}$ and subsets $\mc{P}'_* \subseteq \mc{\mc{P}'}$ and $\mc{P}''_* \subseteq\mc{P}''$, we conclude that (P6) fails with probability at most $$15 \cdot 2^{2l} \cdot  \left(\frac{1}{2}\right)^{k^{3/2}/1200} \leq 15 \cdot \left(\frac{1}{2}\right)^{l} \leq \frac{1}{4},$$ as desired. 

Thus there exist partitions $\mc{P}_0,\ldots,\mc{P}_4$ of $X$ satisfying properties (P1)-(P3),(P5) and (P6). We claim that these properties imply (P4). Clearly, this claim implies the lemma.

First, we will show  that the following additional property holds. 
  \begin{itemize}
  	\item[(P7)] Let $S \subseteq X$ be such that $S$ is $\mc{P}$-tame and $|S| > k/2$ then  $S$ is covered by $\mc{P}$. 
  \end{itemize}
By considering pairs containing arbitrary $x \in S$, we deduce that $|S \cap P'| \geq k/12-1$ for some $P' \in \mc{P}$. Suppose  that there exists $y \in S - P$. By considering the pairs consisting of $y$ and an element of $S \cap P'$ we deduce that there exists $P'' \in \mc{P}, P'' \neq P'$ such that $|P' \cap P''| \geq k/72-2 > \sqrt k$, contradicting (P5). Thus $S \subseteq P'$, and (P7) holds.

Define a weight function $w: X \to \bb{R}_+$ by setting $w(x)=1$ for $x \in X - Q_0$ and $w(x)=k^2/l$ for $x \in Q_0$. Note that $w(P)=k$ for every $P \in \mc{P}$, and $w(X)=kl$.

Let $Z,\mc{P}_*$ be as in (P4). Then $\sum_{P \in \mc{P}_*}w(P) = kl - w(Z)$.  Suppose that there exists $P \in \mc{P}_*$ with $w(P) > k$. Then $|P| \leq k/2$ by (P7). Moreover, $|P \cap Q_0| \leq 2$ as $P$ is $\mc{P}$-tame. Therefore $w(P) \leq 2k^2/l + k/2 < k$, a contradiction.
Thus $w(P) \leq  k$ for every  $P \in \mc{P}_*$, and so   

 \begin{equation}\label{e:pweight}
\sum_{P \in \mc{P}_{**}}(k - w(P)) \leq w(Z) \leq \frac{9k^2}{l} 
\end{equation}
for every $\mc{P}_{**} \subseteq \mc{P}_{*} $.
In particular, we have $$|P| \geq w(P)-\frac{2k^2}{l} \geq k - \frac{11k^2}{l} \geq \frac{9k}{10}$$ for every $P \in  \mc{P}_*$.

 By (P7) there exists $\mc{P}'\in \{\mc{R}, \mc{P}_0, \ldots ,\mc{P}_4\}$ such that at least $l/6$ elements of $\mc{P}_*$ are covered by $\mc{P}'$. If every element of $\mc{P}_*$ is covered by $\mc{P}'$ then (P4) holds.
 
 Thus we assume that there exists  $S \in \mc{P}_*$ such that $S$ is uncovered by $\mc{P}'$. Let $\mc{P}'_*$ be the set of all $P' \in \mc{P'}$ such that there exists $P \in \mc{P}_*$ so that $P \subseteq P'$, and let 
  $$\mc{L} = \{(P,P') \: | \: P \in \mc{P}_*,P'\in \mc{P}'_*, P \subseteq P'  \}.$$ Then \begin{align*}
  \sum_{P' \in \mc{P}'_*}&|P' \cap S| \leq \sum_{(P,P') \in \mc{L}}|P' \cap S|\leq \sum_{(P,P') \in \mc{L}}|P' - P| \leq \sum_{(P,P') \in \mc{L}}(k-w(P)) \leq \frac{9k^2}{l} \leq \frac{k}{5} ,
  \end{align*}
where the penultimate inequality holds by (\ref{e:pweight}).
  Therefore at least $4k/5$ elements of $S$ are covered by an element of $\mc{P}' - \mc{P}'_*$.
  
 By (P5) it follows that $|\mc{P}' - \mc{P}'_*| \geq \sqrt{k}/2$, implying that there exists $$\mc{P}''  \in \{\mc{R}, \mc{P}_0, \ldots ,\mc{P}_4\} - \{\mc{P}'\}$$ such that  at least $\sqrt{k}/{10}$ elements of $\mc{P}_*$ are covered by  $\mc{P}''$. Define  $\mc{P}''_*$ to be the set of all elements of $\mc{P''}$  which contain an element of $\mc{P}_*$. Let $$ I= (\cup_{P' \in \mc{P}'_*}P')  \cap (\cup_{P'' \in \mc{P}''_*}P'' ).$$
 Then $|I| \geq |\mc{P}'|-l/10 \geq l/15$ as at most $l/10$ elements of $\mc{P}'_*$ are disjoint from $\cup_{P'' \in \mc{P}''_*}P'' $ by (P6).  Let $\mc{P}_{**}$ be the set of elements of $\mc{P}_*$ covered by $\mc{P}'_* \cup \mc{P}''_*$.  Repeating the calculations in the previous paragraph with $\mc{P}'_* \cup \mc{P}''_*$ instead of $\mc{P}'_*$, and $I$ instead of $S$, we obtain $$\sum_{P \in \mc{P}}(k-w(P)) \geq |I| \geq l/15,$$ in contradiction to (\ref{e:pweight}).
 \end{proof}

\subsection{Constructing $P_3$-jumbles}\label{s:P3construction}

Let $l$ be a positive integer.
Let $X$ be a finite set with $|X|=l^2$, and let $\mc{R}, \mc{C}, \mc{D}$ be a triple of partitions of $X$ such that $|\mc{R}|=|\mc{C}|=|\mc{D}|=l$, and each part of one of these partitions has exactly one element in common with each part of every other partition. We say that the triple $\mc{B}=(\mc{R},\mc{C},\mc{D})$ is an \emph{$l$-square on $X$}, and we refer to elements of $\mc{R},\mc{C},\mc{D}$ as \emph{rows},\emph{columns} and \emph{diagonals} of $\mc{B}$, respectively, and to all elements of $\mc{R} \cup \mc{C} \cup \mc{D}$  as \emph{lines} of $\mc{B}$. Note that $l$-squares exist for every $l$, indeed they are just different representations of Latin squares.

We use $l$-squares as the building blocks of a more involved construction. An \emph{$l$-pattern} $\mc{Z}$ is a tuple $(P,\mc{B}_1,\ldots,\mc{B}_l)$  such that $P$ is a graph isomorphic to $P_3$, and $\mc{B}_i$ is an $l$-square on a set $X_i$ for every $1 \leq i \leq l$, such that $V(P)$ and $X_1,\ldots,X_l$ are pairwise vertex disjoint. Let $V(\mc{Z})=V(P) \cup X_1 \cup \ldots \cup X_l$. In particular, $|V(\mc{Z})|=l^3+3$. The \emph{lines} of $\mc{Z}$ are the lines of its squares,

Let $G$ be a graph, and let $\mc{B}$ be an $l$-square on $X \subseteq V(G)$. We say that $\mc{B}$ induces \emph{an $(l,r)$-square} in $G$ for an integer $0 \leq r \leq l$ if the rows of $\mc{B}$ and  $r$ diagonals of $\mc{B}$ induce independent sets in $G$, while the columns of $\mc{B}$ and the remaining $l-r$ diagonals induce cliques.

Finally, given an $l$-pattern $\mc{Z}=(P,\mc{B}_1,\ldots,\mc{B}_l)$ we construct a random graph $G=\bf{G}(\mc{Z})$ with $V(G)=V(\mc{Z})$ as follows.  We say that a pair of vertices of $\{u,v\}$ of $V(\mc{Z})$ is \emph{fixed} if either $\{u,v\} \subseteq V(P)$ or $\{u,v\} \subseteq L$ for some  line of $\mc{Z}$, and we say that $\{u,v\}$ is \emph{free}, otherwise.  The intersection of $E(G)$ with the set of fixed pairs is defined deterministically as follows. We require that $G[V(P)]$ coincides with $P$, and that $\mc{B}_i$ induces an $(l,i)$-square in $G$ for every $i$.  Each free pair of vertices forms an edge with probability $1/2$ independently at random. The following lemma is the main result of this section and immediately implies Theorem~\ref{t:existence}.

\begin{lem} Let $\mc{Z}=(P,\mc{B}_1,\ldots,\mc{B}_l)$ be an $l$-pattern. Then the random graph $G=\bf{G}(\mc{Z})$ is a.a.s. an $(P_3,l^2+1)$-jumble.
	
\end{lem}	

\begin{proof}
First, using the fixed part of $G$, we show that for every $0 \leq s \leq l^2$ there exists a partition of $G \setminus V(P)$ into $l^2$ parts, inducing $s$ stable sets and $l^2 - s$ cliques, thus verifying that the condition (S1) in the definition of $(P_3,l^2+1)$-jumble always holds. Indeed, let $s = ql+r$ for some $0 \leq q \leq l$ and $0 \leq r \leq l-1$. If $q=l$ then $r=0$ and the desired partition consists of rows of all the squares of $\mc{Z}$. Otherwise, we form the partition by taking the diagonals of $\mc{B}_r$, the rows of $q$ other squares of $\mc{Z}$ and the columns of the remaining squares.

We say that $G$ is \emph{regular} if every $X \subseteq V(\mc{Z})$ such that $G[X]$ is $P_3$-free satisfies one of the following 
\begin{enumerate}
	\item[(X1)] $X$ is a subset of a line of $\mc{Z}$,
	\item[(X2)] $|X| \leq 3l/5$ and $X-\{x\}$ is a  subset of a line of $\mc{Z}$ for some $x \in X$,
	\item[(X3)] $|X| \leq 2l/5$.
\end{enumerate}
Note that if $G$ is regular then $G$ satisfies the condition (S2) in the definition of $(P_3,l^2+1)$-jumble. Indeed, suppose for a contradiction that  $X_1,X_2,\ldots,X_{l^2+1}$ is a partition of $V(G)$ such that $G[X_i]$ is  $P_3$-free for every $i$. Then $V(P)$ intersects at least two distinct part of the partition, and thus we suppose without loss of generality that either $|V(P) \cap X_1| \geq 2$ and  $|V(P) \cap X_2| \geq 1$, or  $|V(P) \cap X_i| = 1$ for $i=1,2,3$. It follows from regularity of $G$ that $|X_1|+|X_2|+|X_3| \leq 2l$ in both cases, and $|X_j| \leq l$ for every $j \geq 4$. Thus $\sum_{i=1}^{k+1}|X_i| \leq l^3<|V(G)|$, a contradiction.

It remains to show that a.a.s. $G$ is regular. We say that $S \subseteq V(G)$ is \emph{$v$-diverse} for $v \in V(G)-S$ if $v$ has at least two neighbors and at least two non-neighbors in $S$. A standard application of the Chernoff and union bounds implies that a.a.s. $G$  satisfies the following condition
\begin{enumerate}
	\item[$(\star)$] Let $X$ be a subset of a line $L$ of $\mc{Z}$ and let $v_1,v_2  \in  V(G) -L$ be distinct. If $|X|\geq 4l/7$ then $X$ is $v_1$-diverse, and if $|X|\geq l/3$ then $X$ is either $v_1$-diverse or $v_2$-diverse.
\end{enumerate}
Thus we assume that $(\star)$ holds for $G$. Note that if $X$ is a subset of a line $L$, and $X$ is $v$-diverse, then $G[X \cup \{v\}]$ is not $P_3$-free.
 
Let $r=\lceil 2l/5 \rceil$. 
We will show that  for fixed $X \subseteq V(\mc{Z})$ with $|X|=r$ such that $|X \cap L| < l/3$  for every line $L$ the probability that $G[X]$ is $P_3$-free is at most $\left(\frac{3}{4}\right)^{l^2/500}$. The union bound implies that a.a.s.  $G[X]$ is not $P_3$-free for any such $X$. Combining this with $(\star)$ we deduce that $G$ is a.a.s. regular.

It remains to verify the above bound on the probability that $G[X]$ is $P_3$-free. We say that an ordered triple $(u,v_1,v_2)$ of vertices of $X$  is a \emph{seagull with wings $\{u,v_1\}$ and $\{u,v_2\}$} if $\{u,v_1\}$ and $\{u,v_2\}$ are free. A \emph{colony} $\mc{C}$ is a collection of seagulls, such that no wing of a seagull is a subset of the vertex set of another seagull. Suppose that for all free pairs of vertices of  $\mc{Z}$ except for the wings of the seagulls in $\mc{C}$ we determined whether they belong to $G$  or not. Conditioned on this event the probability that a given seagull in $\mc{C}$ induces a   $P_3$ in $G$ is at least $1/4$, and the corresponding events are independent for all the seagulls in a colony. Thus  it suffices to find a colony of size $l^2/500$.

Suppose first that $|X \cap L_0| \geq l/11$ for some line $L_0$ of $G$. Let $Y=X  - L_0$. Then $|Y| \geq l/15$ and for every vertex $y \in Y$ there exist at most two vertices in $x \in |X \cap L_0|$ such that $\{x,y\}$ is fixed. Thus for each such $y$ we can find a colony of at least $l/25$ seagulls with wings sharing the vertex $y$ and otherwise disjoint from all the remaining vertices in $X \cap L_0$. Taking the union of such colonies over all $y \in Y$ produces the required colony.

Suppose now that $|X \cap L| \leq l/11$  for every line $L$ of $G$. Choose arbitrary $Y \subseteq X$ with $|Y|=\lfloor l/15 \rfloor$. Then  for every vertex $y \in Y$ there exist at least $2l/33$ vertices in $X-Y$  which form a free pair with $y$. We can now repeat the argument in the previous paragraph.
\end{proof}

\bibliographystyle{alpha}
\bibliography{snorin}

\end{document}